\documentclass{amsart}
\usepackage[top=2cm,bottom=2cm,right=2.5cm,left=2.5cm]{geometry}
\usepackage{amssymb}
\usepackage{amsmath, amsthm, amscd, amsfonts, amssymb, graphicx, color}
\usepackage[bookmarksnumbered, colorlinks, plainpages]{hyperref}
\hypersetup{colorlinks=true,linkcolor=red, anchorcolor=green, citecolor=cyan, urlcolor=red, filecolor=magenta, pdftoolbar=true}
 
\usepackage{hyperref}

\newtheorem{theorem}{Theorem}[section]
\newtheorem{lemma}[theorem]{Lemma}
\newtheorem{proposition}[theorem]{Proposition}
\newtheorem{corollary}[theorem]{Corollary}
\theoremstyle{definition}
\newtheorem{definition}[theorem]{Definition}

\theoremstyle{remark}

\numberwithin{equation}{section}

\begin{document}
	
	\setcounter{page}{1}
	
	\title[Weaving continuous generalized frames for operators]{Weaving continuous generalized frames for operators}
	
\author[H. Massit, M. Rossafi, C.  Park]{Hafida Massit$^{1}$, Mohamed Rossafi$^{2}$ and Choonkil Park$^{3*}$}

\address{$^{1}$Department of Mathematics, Faculty Of Sciences, University of Ibn Tofail, Kenitra, Morocco}
\email{\textcolor[rgb]{0.00,0.00,0.84}{massithafida@yahoo.fr}}
\address{$^{2}$Department of Mathematics,  Faculty of Sciences, Dhar El Mahraz University Sidi Mohamed Ben Abdellah, Fes, Morocco}
\email{\textcolor[rgb]{0.00,0.00,0.84}{rossafimohamed@gmail.com; mohamed.rossafi@usmba.ac.ma}}
\address{$^{3}$Research Institute for Natural Sciences, Hanyang University, Seoul 04763,  Korea}
\email{\textcolor[rgb]{0.00,0.00,0.84}{baak@hanyang.ac.kr}}

	
	
	\date{
		\newline \indent $^{*}$Corresponding author:  Choonkil Park (email: baak@hanyang.ac.kr, orcid: 0000-0001-6329-8228, fax: +82-2-2281-0019).}

	\begin{abstract} Recently, Bemrose et al. \cite{BE}  developed a theory of weaving frames, which  was motivated by a problem regarding distributed signal processing. In this present article, we introduce the atomic $g$-system and we generalize some of the known results in continuous $L$-frames, weaving continuous and weaving continuous $ g$-frames, also we study weaving continuous $ L$-$g$-frames in Hilbert spaces. 
	Moreover, we study the behaviour continuous $ L$-$g$-frames under some perturbations, and  we show that approximate $L$-duals are stable under small perturbation and that it is possible to remove some elements of a woven continuous $ L$-$g$-frame and still have a woven continuous $ L$-$g$-frame.
\end{abstract}

	\maketitle

	\section{Introduction and preliminaries}

	In 1952, Duffin and Schaffer \cite{Duf} introduced the concept of frames in Hilbert spaces to study some problems in nonharmonic Fourier series, their importance in data processing was reintroduced by Daubechies, Grossman and Meyer \cite{DGM}, the theory of frames have found many applications in engineering, but are also important tools in pure mathematics. Frames play key roles in wavelet theory and Gabor frames. The
	concept of a generalization of frames to a family indexed by some locally compact space
	endowed with a Radon measure was proposed by Ali, Antoine and Gazeau \cite{GAZ}. These frames are known as continuous frames. Gabrado and Han  \cite{GHN} called these ``frames associated with measurable spaces". For more about frames and generalizations, see \cite{cc, ns, FR1,  gs, massit, po,   r5, r6, r3, r1,  RFDCA}.

 Motivated by a problem regarding distributed signal processing, Bemrose et al. \cite{BE}  introduced a new concept of weaving frames in separable Hilbert space. The fundamental properties of weaving frames were examined by Casazza and Lynch in \cite{Cas}. Weaving frames were further studied by Casazza, Freeman and Lynch \cite{P.G}.
 
 In this paper, we generalize some results in \cite{K} to the continuous version of weaving continuous $L$-$g$-frames for operators in Hilbert spaces. In 2012, G\v{a}vruta introduced new kinds of frames for operators (or
 $ L$-frames), while studying the atomic systems with respect to a bounded operator $ L $.

Throughout this paper, we suppose  that $ \mathcal{H} $ is a separable Hilbert space, $\{ \mathcal{H}_{\varsigma}, \varsigma \in \mathfrak{A} \}$ is a sequence of separable Hilbert spaces, $ I $  is the identity operator on $ \mathcal{H}$,
	 $(\mathfrak{A},\mu)$  is  a measure space with positive measure $\mu$. 
	Also, for every $ \varsigma\in\mathfrak{A} $, $ B(\mathcal{H}, \mathcal{H}_{\varsigma}) $ is the set of all bounded linear operators from $ \mathcal{H} $ to $ \mathcal{H}_{\varsigma} $, let $ L \in B(\mathcal{H})$, with closed range and $ GL(\mathcal{H}) =\{L\in B(\mathcal{H}): \;L \; is \;invertible\}$.

A family of vectors $F= \{F_{\varsigma}\}_{\varsigma\in \mathfrak{A}} $ in a separable Hilbert $ \mathcal{H}$ is said to be a continuous frame if there exist constants $ 0< A \leq B <\infty $ such that 
\begin{equation}\label{eq1}
	A \Vert f\Vert^{2}\leq \int_{\mathfrak{A}} \vert \langle f,  F_{\varsigma} \rangle \vert^{2} d\mu(\varsigma) \leq B \Vert f\Vert^{2}, \;\forall f\in\mathcal{H},
\end{equation}
and then the constants $ A $ and $ B $ are called frame bounds. 

The family $ \{F_{\varsigma}\}_{\varsigma \in \mathfrak{A}} $ is said to be a Bessel sequence whenever in (\ref{eq1}), the right hand side holds. In the case of $ A=  B =1  $, $ \{F_{\varsigma}\}_{\varsigma \in \mathfrak{A}} $ is  called a Parseval frame. If  $ A=  B  $, then  it is called a tight frame.

A family of vectors $F= \{F_{\varsigma}\}_{\varsigma\in \mathfrak{A}} $ in a separable Hilbert $ \mathcal{H}$ is said to be a continuous $ L- $frame if there exist constants $ 0< A\leq B <\infty $ such that 
\begin{equation*}\label{eq}
	A \Vert L^{*}f\Vert^{2}\leq \int_{\mathfrak{A}} \vert \langle f,  F_{\varsigma} \rangle \vert^{2} d\mu(\varsigma) \leq B \Vert f\Vert^{2}, \;\forall f\in\mathcal{H}.
\end{equation*}
A sequence $ \chi= \{\chi_{\varsigma} \in B(\mathcal{H},\mathcal{H}_{\varsigma}) ,\;\varsigma \in \mathfrak{A}\} $
is called a $ g- $frame for $ \mathcal{H} $ with respect to $ \{\mathcal{H}_{\varsigma},\; \varsigma \in\mathfrak{A}\} $ if there exist $ 0< A\leq B <\infty $ such that for every $ f\in \mathcal{H} $
\begin{equation}\label{2}
	A \Vert f\Vert^{2}\leq \int_{\mathfrak{A}} \Vert  \chi_{\varsigma}f \Vert^{2} d\mu(\varsigma) \leq B \Vert f\Vert^{2}, \;\forall f\in\mathcal{H},
\end{equation}
and then $ A,B $ are called $ g- $frame bounds. The family $ \{\chi_{\varsigma}\}_{\varsigma \in \mathfrak{A}} $ is said to be a $ g-$Bessel sequence whenever in (\ref{2}), the right hand side holds. In the case of $ A =  B =1  $, $ \{\chi_{\varsigma}\}_{\varsigma \in \mathfrak{A}} $  is called a continuous Parseval $ g- $ frame. If $ A =  B  $, then  it is called a tight continuous $ g- $frame.

For every sequence $ \{\mathcal{H}_{\varsigma}\}_{\varsigma\in\mathfrak{A}} $,  let  
\begin{equation*}
(	\int_{\mathfrak{A}} \bigoplus \mathcal{H}_{\varsigma}d\mu(\varsigma))_{L^{2}}= \{(F_{\varsigma})_{\varsigma\in\mathfrak{A}},\; F_{\varsigma} \in \mathcal{H}_{\varsigma},\; \int_{\mathfrak{A}}\Vert F_{\varsigma}\Vert ^{2} d\mu(\varsigma)< \infty\},
\end{equation*}
with pointwise operations and then  the following inner product space is a Hilbert space 
\begin{equation*}
	\langle (F_{\varsigma})_{\varsigma\in\mathfrak{A}}, (G_{\varsigma})_{\varsigma\in\mathfrak{A}}\rangle = \int_{\mathfrak{A}}\langle F_{\varsigma},G_{\varsigma}\rangle d\mu(\varsigma).
	\end{equation*}
If $ \chi $ is a continuous $ g- $Bessel sequence, then the synthesis operator for $ \chi $ is the linear operator 
\begin{equation*}
	T_{\chi}: (	\int_{\mathfrak{A}} \bigoplus \mathcal{H}_{\varsigma}d\mu(\varsigma))_{L^{2}} \rightarrow \mathcal{H},\;\;\;\; T_{\chi}(F_{\varsigma})_{\varsigma\in\mathfrak{A}}=\int_{\mathfrak{A}}\chi_{\varsigma}^{\ast}F_{\varsigma}d\mu(\varsigma).
\end{equation*}
The adjoint of the synthesis operator is called the analysis operator and is defined by 
\begin{equation*}
	T_{\chi}^{\ast}:\mathcal{H}\rightarrow (	\int_{\mathfrak{A}} \bigoplus \mathcal{H}_{\varsigma}d\mu(\varsigma))_{L^{2}}  ,\;\;\;\; T_{\chi}^{\ast}(f)=(\chi_{\varsigma}^{\ast}f)_{\varsigma\in \mathfrak{A}}.
\end{equation*}
We call $ S_{\chi}= T_{\chi}T_{\chi}^{\ast} $ the $ g- $frame operator of $ \chi $ and $ S_{\chi}f= \int_{\mathfrak{A}}\chi_{\varsigma}^{\ast}\chi_{\varsigma}f d\mu(\varsigma), \;\;f\in \mathcal{H} $.

If $ \chi= (\chi_{\varsigma})_{\varsigma \in \mathfrak{A}}$ is a $ g- $frame with lower and upper $ g- $frame bounds $ A,\;B $, respectively, then the $ g- $frame operator of $ \chi $ is a bounded, positive and invertible operator on $ \mathcal{H} $ and 
\begin{equation*}
	A\langle f,f\rangle \leq \langle S_{\chi} f,f \rangle \leq B\langle f,f\rangle,\;\;f\in\mathcal{H},
\end{equation*}
and so 
\begin{equation*}
	A\cdot I\leq S_{\chi} \leq B\cdot I.
\end{equation*}
Let $ L \in B(\mathcal{H}) $. A sequence $ \chi= \{\chi_{\varsigma} \in B(\mathcal{H},\mathcal{H}_{\varsigma}), \varsigma \in \mathfrak{A} \} $ is called a continuous $ L-g- $frame if there exist constants $ 0<A\leq B<\infty $ such that 
\begin{equation}
A \Vert L^{\ast}f\Vert^{2}\leq \int_{\mathfrak{A}} \Vert  \chi_{\varsigma}f \Vert^{2} d\mu(\varsigma) \leq B \Vert f\Vert^{2}, \;\forall f\in\mathcal{H}.
\end{equation}

\section{Main results}

In \cite{K},  Khosravi and Banyarani defined the atomic $ g- $system for an operator $ L\in B(\mathcal{H}) $ and by using this idea we introduce the following definition. 

\begin{definition} Let $ L\in B(\mathcal{H}) $. A sequence $ \{\chi_{\varsigma}\in B(\mathcal{H},\mathcal{H}_{\varsigma}),\varsigma\in\mathfrak{A}\} $ is called an atomic $ g- $system for $ L $ if the following hold: 
	\begin{itemize}
		\item [(1)] $\{\chi_{\varsigma}\}_{\varsigma\in\mathfrak{A}}   $ is a continuous $ g- $Bessel sequence.
		\item[(2)]For any $ f\in\mathcal{H} $, there exists $\mathcal{W}_{f}= (w_{\varsigma})_{\varsigma}\in (\int_{\mathfrak{A}} \bigoplus \mathcal{H}_{\varsigma}d\mu(\varsigma))_{L^{2}} $ such that $ Lf= \int_{\mathfrak{A}}\chi_{\varsigma}^{\ast}(w_{\varsigma})d\mu(\varsigma) $, where $ \Vert \mathcal{W}_{f} \Vert \leq \alpha \Vert f\Vert$ and $ \alpha $ is a positive constant.
	\end{itemize}  
	\end{definition}

\begin{definition} Let $ \chi=\{\chi_{\varsigma}\in B(\mathcal{H},\mathcal{H}_{\varsigma}),\varsigma\in\mathfrak{A}\} $ and $ \xi=\{\xi_{\varsigma}\in B(\mathcal{H},\mathcal{H}_{\varsigma}), \varsigma\in \mathfrak{A} \} $  be two continuous $ g- $frames for $ \mathcal{H} $. We call $ \{\chi_{\varsigma}\}_{\varsigma \in\mathfrak{A}} $
	and $ \{\xi_{\varsigma}\}_{\varsigma \in\mathfrak{A}} $ woven continuous $ g- $frames if there exist	$ 0< A\leq B <\infty $ such that for every $ \mathcal{J}\subset \mathfrak{A} $ and  \for all $ f\in\mathcal{H} $ 
	\begin{equation*}
		A \Vert f\Vert^{2}\leq \int_{\mathcal{J}} \Vert  \chi_{\varsigma}f \Vert^{2} d\mu(\varsigma)+ \int_{\mathcal{J}^{C}} \Vert  \xi_{\varsigma}f \Vert^{2} d\mu(\varsigma)\leq B \Vert f\Vert^{2}.	
	\end{equation*}
	We  say that $ \{\chi_{\varsigma}\}_{\varsigma \in\mathfrak{A}} $
	and $ \{\xi_{\varsigma}\}_{\varsigma \in\mathfrak{A}} $ are  $ (A,B)- $woven 
continuous $ g- $frames.
\end{definition}

\begin{definition} Let $ L\in B(\mathcal{H}) $, $ \chi=\{\chi_{\varsigma}\in B(\mathcal{H},\mathcal{H}_{\varsigma}),\varsigma\in\mathfrak{A}\} $ and $ \xi=\{\xi_{\varsigma}\in B(\mathcal{H},\mathcal{H}_{\varsigma}), \varsigma\in \mathfrak{A} \} $  be two continuous $L- g- $frames for $ \mathcal{H} $. We call $ \{\chi_{\varsigma}\}_{\varsigma \in\mathfrak{A}} $
	and $ \{\xi_{\varsigma}\}_{\varsigma \in\mathfrak{A}} $ woven continuous $ g- $frames if there exist	$ 0< A\leq B <\infty $ such that for every $ \mathcal{J}\subset \mathfrak{A} $ and  \for all $ f\in\mathcal{H} $ 
	\begin{equation*}
		A \Vert L^{\ast}f\Vert^{2}\leq \int_{\mathcal{J}} \Vert  \chi_{\varsigma}f \Vert^{2} d\mu(\varsigma)+ \int_{\mathcal{J}^{C}} \Vert  \xi_{\varsigma}f \Vert^{2} d\mu(\varsigma)\leq B \Vert f\Vert^{2}.	
	\end{equation*}
	We  say that $ \{\chi_{\varsigma}\}_{\varsigma \in\mathfrak{A}} $
	and $ \{\xi_{\varsigma}\}_{\varsigma \in\mathfrak{A}} $ are  $ (A,B)- $woven continuous $L- g- $frames.
\end{definition}
We try to generalize some of the known results in continuous $ L- $frames, 
weaving continuous frames and weaving continuous $ g- $frames and also we study weaving continuous $ L-g- $frames.

\begin{definition} Let $ L\in B(\mathcal{H}) $. Sequences $ \{\chi_{\varsigma}\} _{\varsigma\in\mathfrak{A}} $ and $ \{\xi_{\varsigma}\} _{\varsigma\in\mathfrak{A}} $ are  called continuous atomic $ g- $systems for $ L $ if the following hold: 
	\begin{itemize}
		\item [(1)] $ \{\chi_{\varsigma}\} _{\varsigma\in\mathfrak{A}} $ and $ \{\xi_{\varsigma}\} _{\varsigma\in\mathfrak{A}} $  are continuous $ g- $Bessel sequences.
		\item[(2)]For any $ f\in\mathcal{H} $ and any $ \mathcal{J}\subset \mathfrak{A} $, there exist $\mathcal{W}_{f}= (w_{\varsigma})_{\varsigma},\;\mathcal{W'}_{f}= (w'_{\varsigma})_{\varsigma} \in (\int_{\mathfrak{A}} \bigoplus \mathcal{H}_{\varsigma}d\mu(\varsigma))_{L^{2}} $ such that 
	$$	 Lf=  \int_{\mathcal{J}}   \chi_{\varsigma}^{\ast}(w_{\varsigma})  d\mu(\varsigma)+ \int_{\mathcal{J}^{C}}   \xi_{\varsigma}^{\ast}(w'_{\varsigma})  d\mu(\varsigma) $$  
		 with $ \Vert \mathcal{W}_{f} \Vert \leq \alpha_{1} \Vert f\Vert$ and $ \Vert \mathcal{W'}_{f} \Vert \leq \alpha_{2} \Vert f\Vert$, ( $ \alpha_{1}\geq 0 $ and $ \alpha_{2} \geq 0 $).
	\end{itemize}  
	\end{definition} 

\begin{theorem} Let  $ \{\chi_{\varsigma}\in B(\mathcal{H},\mathcal{H}_{\varsigma}),\;{\varsigma\in\mathfrak{A}} \}  $ and $ \{\xi_{\varsigma}\in B(\mathcal{H},\mathcal{H}_{\varsigma}),\;{\varsigma\in\mathfrak{A}} \}  $ be a continuous woven atomic $ g- $system for $ L $. Then $ \{\chi_{\varsigma}\in B(\mathcal{H},\mathcal{H}_{\varsigma}),\;{\varsigma\in\mathfrak{A}} \}  $ and $ \{\xi_{\varsigma}\in B(\mathcal{H},\mathcal{H}_{\varsigma}),\;{\varsigma\in\mathfrak{A}} \}  $ are woven $ L-g- $frames. 
	\end{theorem}

\begin{proof} Let $ f\in\mathcal{H} $. For every $ g\in \mathcal{H} $ with $ \Vert g\Vert =1 $ and every $ \mathcal{J} $, there exist  $ (w_{\varsigma})_{\varsigma},\; (w'_{\varsigma})_{\varsigma} \in (\int_{\mathfrak{A}} \bigoplus \mathcal{H}_{\varsigma}d\mu(\varsigma))_{L^{2}} $ such that
$	 Lg=  \int_{\mathcal{J}}   \chi_{\varsigma}^{\ast}(w_{\varsigma})  d\mu(\varsigma)+ \int_{\mathcal{J}^{C}}   \xi_{\varsigma}^{\ast}(w'_{\varsigma})  d\mu(\varsigma) $. Then 
\begin{align*}
 \Vert L^{\ast}f \Vert&= sup_{\Vert g \Vert =1} \vert \langle L^{\ast}f,g\rangle \vert = sup_{\Vert g \Vert =1} \vert \langle f, \int_{\mathfrak{J}}   \chi_{\varsigma}^{\ast}(w_{\varsigma})  d\mu(\varsigma)+ \int_{\mathcal{J}^{C}}   \xi_{\varsigma}^{\ast} w'_{\varsigma}  d\mu(\varsigma)\rangle \vert\\
&\leq sup_{\Vert g \Vert =1} \vert \langle f,\int_{\mathcal{J}}   \chi_{\varsigma}^{\ast} w_{\varsigma}  d\mu(\varsigma)\rangle \vert +sup_{\Vert g \Vert =1} \vert \langle f,\int_{\mathcal{J^{C}}}   \xi_{\varsigma}^{\ast} w'_{\varsigma} d\mu(\varsigma)\rangle \vert\\
&=   sup_{\Vert g \Vert =1} \vert  \int_{\mathcal{J}}\langle   \chi_{\varsigma}f,w_{\varsigma} \rangle d\mu(\varsigma)\vert +sup_{\Vert g \Vert =1} \vert \int_{\mathcal{J^{C}}}  \langle  \xi_{\varsigma}f,w'_{\varsigma}\rangle  d\mu(\varsigma) \vert  \\
&\leq sup_{\Vert g \Vert =1} (  \int_{\mathcal{J}}  \Vert\chi_{\varsigma}f\Vert ^{2}  d\mu(\varsigma))^{\frac{1}{2}}(  \int_{\mathcal{J}}  \Vert w_{\varsigma}\Vert ^{2}  d\mu(\varsigma))^{\frac{1}{2}} +sup_{\Vert g \Vert =1} (  \int_{\mathcal{J^{C}}}  \Vert\xi_{\varsigma}f\Vert ^{2}  d\mu(\varsigma))^{\frac{1}{2}}(  \int_{\mathcal{J^{C}}}  \Vert w'_{\varsigma}\Vert ^{2}  d\mu(\varsigma))^{\frac{1}{2}} \\
&\leq sup_{\Vert g \Vert =1} (  \int_{\mathcal{J}}  \Vert\chi_{\varsigma}f\Vert ^{2}  d\mu(\varsigma)+(\int_{\mathcal{J^{C}}}  \Vert\xi_{\varsigma}f\Vert ^{2}  d\mu(\varsigma))^{\frac{1}{2}}[( \int_{\mathfrak{A}}  \Vert w_{\varsigma}\Vert ^{2}  d\mu(\varsigma))^{\frac{1}{2}} +(  \int_{\mathfrak{A}}  \Vert w'_{\varsigma}\Vert ^{2}  d\mu(\varsigma))^{\frac{1}{2}} ]\\
& \leq (\alpha_{1}+\alpha_{2})sup_{\Vert g \Vert =1}\Vert g\Vert ( \int_{\mathcal{J}}  \Vert\chi_{\varsigma}f\Vert ^{2}  d\mu(\varsigma)+\int_{\mathcal{J^{C}}}  \Vert\xi_{\varsigma}f\Vert ^{2}  d\mu(\varsigma))^{\frac{1}{2}}.
\end{align*}
Therefore, $ \int_{\mathcal{J}}  \Vert\chi_{\varsigma}f\Vert ^{2}  d\mu(\varsigma)+\int_{\mathcal{J^{C}}}  \Vert\xi_{\varsigma}f\Vert ^{2}  d\mu(\varsigma)  \geq \dfrac{1}{(\alpha_{1}+\alpha_{2})^{2}}\Vert L^{\ast}f \Vert ^{2}$. 	
\end{proof}

Let $ \chi=\{\chi_{\varsigma}\}_{\varsigma \in\mathfrak{A}} $ and  $\xi= \{\xi_{\varsigma}\}_{\varsigma \in\mathfrak{A}} $
	be continuous $ g- $Bessel sequences, with bounds $ B,\;B' $, respectively. Then the operator $ S_{\xi,\chi}:\mathcal{H}\rightarrow \mathcal{H}  $ defined by 
	\begin{equation*}
		S_{\xi,\chi}(f)= T_{\xi}T_{\chi}^{\ast}(f)= \int_{\mathfrak{A}}\xi_{\varsigma}^{\ast}\chi_{\varsigma}(f)d\mu(\varsigma), \;\;\;f\in\mathcal{H}
	\end{equation*}
is a bounded linear operator with $ \Vert S_{\xi,\chi} \Vert \leq \sqrt{BB'}$, $ S_{\xi,\chi}= S_{\xi,\chi}^{\ast}$ and $ S_{\xi}=S_{\xi,\xi} $.

\begin{lemma}  Let $ \chi=\{\chi_{\varsigma}\}_{\varsigma \in\mathfrak{A}} $ and  $\xi= \{\xi_{\varsigma}\}_{\varsigma \in\mathfrak{A}} $
	be continuous $ g- $Bessel sequences. If there exists $ \gamma>0 $ such that  $ S_{\xi,\chi}(f)\geq \gamma \Vert L^{\ast}f\Vert $, then 
$ \chi=\{\chi_{\varsigma}\}_{\varsigma \in\mathfrak{A}} $ and  $\xi= \{\xi_{\varsigma}\}_{\varsigma \in\mathfrak{A}} $ are $ L-g- $frames.	
	\end{lemma}

\begin{proof} Suppose that there exists a number $ \gamma>0 $ such that	
 for all $ f\in \mathcal{H} $, $ S_{\xi,\chi}(f)\geq \gamma \Vert L^{\ast}f\Vert $ and 
then we have 
\begin{align*}
\gamma\Vert L^{\ast} f\Vert 	&\leq \Vert S_{\xi,\chi}(f)\Vert \\
&= sup_{\Vert g \Vert= 1}\vert \langle  \int_{\mathfrak{A}}\xi^{\ast}_{\varsigma}\chi_{\varsigma}(f)d\mu(\varsigma),g\rangle \vert\\
&\leq sup_{\Vert g\Vert =1}(\int_{\mathfrak{A}}\Vert \chi_{\varsigma}(f)\Vert 
^{2}d\mu(\varsigma))^{\frac{1}{2}}(\int_{\mathfrak{A}}\Vert \xi_{\varsigma}(g)\Vert^{2} d\mu (\varsigma) )^{\frac{1}{2}}\\
&\leq \sqrt{B}(\int_{\mathfrak{A}}\Vert \chi_{\varsigma}(f)\Vert^{2}d\mu(\varsigma))^{\frac{1}{2}}.
\end{align*}
Hence
\begin{equation*}
	(\dfrac{\gamma^{2}}{B})\Vert L^{\ast}f\Vert^{2}\leq \int_{\mathfrak{A}}\Vert \chi_{\varsigma}(f)\Vert^{2}d\mu(\varsigma).
\end{equation*}
	
On the other hand, we have $ S_{\xi,\chi}^{\ast} = S_{\chi,\xi}$ and so $( \chi_{\varsigma})_{\varsigma} $ is an $ L-g- $frame.
\end{proof}

 \begin{proposition} Let $\chi=\{\chi_{\varsigma}\in B(\mathcal{H},\mathcal{H}_{\varsigma})\} _{\varsigma\in\mathfrak{A}} $,  $\xi=\{\xi_{\varsigma}\in B(\mathcal{H},\mathcal{H}_{\varsigma})\} _{\varsigma\in\mathfrak{A}} $,  $\chi'=\{\chi'_{\varsigma}\in B(\mathcal{H},\mathcal{H'}_{\varsigma})\} _{\varsigma\in\mathfrak{A}} $
 and  $\xi'=\{\xi'_{\varsigma}\in B(\mathcal{H},\mathcal{H}_{\varsigma})\} _{\varsigma\in\mathfrak{A}} $ be $ g- $Bessel sequences with bounds $ B_{1}, \;B_{2}, \;B_{3}, \;B_{4} $, respectively. If there exists $ \gamma>0 $	such that
  $$   \Vert (S_{\chi,\chi'}^{\mathcal{J}}+S_{\xi,\xi'}^{\mathcal{J}^{C}})f \Vert \geq \gamma \Vert L^{\ast}f\Vert$$   
  for $ \mathcal{J}\subset \mathfrak{A} $ and $ f\in\mathcal{H} $, then $ \{\chi_{\varsigma}^{'}\}_{\varsigma \in\mathfrak{A}} $ and $ \{\xi_{\varsigma}^{'}\}_{\varsigma\in\mathfrak{A}} $ are woven continuous $ L-g- $frames and also  $ \{\chi'_{\varsigma}\}_{\varsigma \in\mathfrak{A}} $ and $ \{\xi'_{\varsigma}\}_{\varsigma\in\mathfrak{A}} $ are woven continuous $ L-g- $frames.
 \end{proposition}

\begin{proof}
	Suppose that there is  $ \gamma $ such that for all $ \mathcal{J}\subset \mathfrak{A} $ and $ f\in \mathcal{H} $
		\begin{equation*}
		\gamma \Vert L^{\ast}f\Vert \leq  \Vert (S_{\chi,\chi'}^{\mathcal{J}}+S_{\xi,\xi'}^{\mathcal{J}^{C}})f \Vert . 
	\end{equation*}
Then 
\begin{align*}
	 \Vert (S_{\chi,\chi'}^{\mathcal{J}}+S_{\xi,\xi'}^{\mathcal{J}^{C}})f \Vert &\leq  \Vert S_{\chi,\chi'}^{\mathcal{J}}f \Vert+\Vert S_{\xi,\xi'}^{\mathcal{J}^{C}}f \Vert\\
	 &= \Vert (T_{\chi}T_{\chi'}^{\ast})^{\mathcal{J}}(f)\Vert +\Vert (T_{\xi}T_{\xi'}^{\ast})^{\mathcal{J}^{C}}(f)\Vert\\
	 &\leq \Vert T_{\chi}\Vert (\int_{\mathcal{J}}\Vert \chi'_{\varsigma}f \Vert^{2} d\mu(\varsigma)  )^{\frac{1}{2}}+\Vert T_{\xi}\Vert (\int_{\mathcal{J}^{C}}\Vert \xi'_{\varsigma}f \Vert^{2} d\mu(\varsigma)  )^{\frac{1}{2}}\\
	 &\leq\sqrt{B_{1}}(\int_{\mathcal{J}}\Vert \chi'_{\varsigma}f\Vert^{2} d\mu(\varsigma))^{\frac{1}{2}}+\sqrt{B_{2}}(\int_{\mathcal{J}^{C}}\Vert \xi'_{\varsigma}f\Vert^{2}d\mu(\varsigma))^{\frac{1}{2}}\\
	 &\leq (\sqrt{B_{1}}+\sqrt{B_{2}}) (\int_{\mathcal{J}}\Vert \chi'_{\varsigma}f\Vert^{2}d\mu(\varsigma) +\int_{\mathcal{J}^{C}}\Vert \xi'_{\varsigma}f\Vert^{2} d\mu(\varsigma))^{\frac{1}{2}}.
	\end{align*} 
Hence
$$  \int_{\mathcal{J}}\Vert \chi'_{\varsigma}f\Vert^{2}d\mu(\varsigma)+ \int_{\mathcal{J}^{C}}\Vert \xi'_{\varsigma}f\Vert^{2}d\mu(\varsigma) \geq \dfrac{\gamma^{2}\Vert L^{\ast}f\Vert ^{2}}{(\sqrt{B_{1}}+\sqrt{B_{2}})^{2}}.$$

On the other hand, since $ S^{\ast}_{\chi,\xi}=S_{\xi,\chi} $,   $(S_{\chi,\chi'}^{\mathcal{J}}+S_{\xi,\xi'}^{\mathcal{J}^{C}})^{\ast} = S_{\chi',\chi}^{\mathcal{J}}+S_{\xi',\xi}^{\mathcal{J}^{C}}$ and we have the result.   
\end{proof}

\begin{theorem} Let $\chi=\{\chi_{\varsigma}\in B(\mathcal{H},\mathcal{H}_{\varsigma})\} _{\varsigma\in\mathfrak{A}} $,  $\xi=\{\xi_{\varsigma}\in B(\mathcal{H},\mathcal{H}_{\varsigma})\} _{\varsigma\in\mathfrak{A}} $ be $ (A,\;B) $ woven continuous $ L-g- $frames and  $\chi'=\{\chi'_{\varsigma}\in B(\mathcal{H}',\mathcal{H}'_{\varsigma})\} _{\varsigma\in\mathfrak{A}} $,  $\xi'=\{\xi'_{\varsigma}\in B(\mathcal{H}',\mathcal{H}'_{\varsigma})\} _{\varsigma\in\mathfrak{A}} $ be $ (A,\;B) $ woven continuous $ L-g- $frames.
\begin{itemize}
	\item[(1)] $ \{\chi_{\varsigma}\oplus \chi'_{\varsigma}\}_{\varsigma} $ are $ (min\{A,\;A'\},\; max\{B,\;B'\} )$ woven continuous $ L-g- $frames.
	\item[(2)]If $ \mathcal{H}= \mathcal{H}',\; \mathcal{H}_{\varsigma}= \mathcal{H}'_{\varsigma} $ for each $ \varsigma \in \mathfrak{A} $ and for $ \mathcal{J}\subset \mathfrak{A} $
		$$ S_{\chi',\chi}^{\mathcal{J}}+S_{\chi,\chi'}^{\mathcal{J}}+S_{\xi',\xi}^{\mathcal{J}^{C}}+S_{\xi,\xi'}^{\mathcal{J}^{C}}\geq 0, $$
	then $ \{\chi_{\varsigma}+\chi_{\varsigma}'\}_{\varsigma} $ and $ \{\xi_{\varsigma}+\xi_{\varsigma}'\}_{\varsigma} $ are woven continuous $ L-g- $frames, where 
	
	$ S_{\xi,\xi'}^{\mathcal{J}^{C}}= \int_{\mathcal{J}^{C}} \xi_{\varsigma}^{\ast}\xi_{\varsigma}'d\mu(\varsigma) $.
\end{itemize}
\end{theorem}

\begin{proof}$ (1) $ Let $ (f,g) $ be an arbitrary element of $ \mathcal{H}\oplus \mathcal{H'} $ and $ \mathcal{J}\subset \mathfrak{A} $. Then
\begin{align*}
&\int_{\mathcal{J}} \Vert ( \chi_{\varsigma}\oplus \chi'_{\varsigma})(f,g)\Vert^{2}d\mu(\varsigma)+\int_{\mathcal{J}^{C}} \Vert ( \xi_{\varsigma}\oplus \xi'_{\varsigma})(f,g)\Vert^{2}d\mu(\varsigma)\\
&= \int_{\mathcal{J}}\Vert (\chi_{\varsigma}f,\chi'_{\varsigma}g)\Vert^{2}d\mu(\varsigma)+\int_{\mathcal{J}^{C}}\Vert (\xi_{\varsigma}f,\xi'_{\varsigma}g)\Vert^{2}d\mu(\varsigma)\\
&=\int_{\mathcal{J}}\langle \chi_{\varsigma}f,\chi'_{\varsigma}g), (\chi_{\varsigma}f,\chi'_{\varsigma}g)\rangle d\mu(\varsigma) +\int_{\mathcal{J}^{C}}\langle \xi_{\varsigma}f,\xi'_{\varsigma}g), (\xi_{\varsigma}f,\xi'_{\varsigma}g)\rangle d\mu(\varsigma)\\
&=\int_{\mathcal{J}}(\Vert \chi_{\varsigma}f\Vert ^{2}+\Vert \chi'_{\varsigma}g\Vert^{2})d\mu(\varsigma) +\int_{\mathcal{J}^{C}}(\Vert \chi_{\varsigma}f\Vert ^{2}+\Vert \chi'_{\varsigma}g\Vert^{2})d\mu(\varsigma)\\
&\leq B\Vert f\Vert^{2} +B'\Vert g\Vert^{2} \\
&\leq max\{ B,\;B'\} \Vert(f,g)\Vert^{2}.
\end{align*}
	
Similarly,  for the lower bound, we can have the result. 

$ (2) $	For every $\mathcal{J}\subset \mathfrak{A} $, we have 
\begin{align*}
S_{\chi'+\chi}^{J}+S_{\xi+\xi'}^{\mathcal{J}^{C}}&= \int_{\mathcal{J}} (\chi_{\varsigma}+\chi_{\varsigma}')^{\ast}(\chi_{\varsigma}+\chi_{\varsigma}')d\mu(\varsigma)+\int_{\mathcal{J}^{C}} (\xi_{\varsigma}+\xi_{\varsigma}')^{\ast}(\xi_{\varsigma}+\xi_{\varsigma}')d\mu(\varsigma)\\
&=\int_{\mathcal{J}} (\chi_{\varsigma})^{\ast}(\chi_{\varsigma})d\mu(\varsigma)+\int_{\mathcal{J}} (\chi'_{\varsigma})^{\ast}(\chi'_{\varsigma})d\mu(\varsigma)+\int_{\mathcal{\mathcal{J}}^{C}} (\xi_{\varsigma})^{\ast}\xi_{\varsigma}d\mu(\varsigma)+\int_{\mathcal{J}^{C}} (\xi_{\varsigma})^{\ast}\xi'_{\varsigma}d\mu(\varsigma)\\
&+\int_{\mathcal{J}} (\chi_{\varsigma})^{\ast}(\chi'_{\varsigma})+ (\chi'_{\varsigma})^{\ast}(\chi_{\varsigma})d\mu(\varsigma)+\int_{\mathcal{J}^{C}} (\xi_{\varsigma})^{\ast}(\xi'_{\varsigma})+ (\xi'_{\varsigma})^{\ast}(\xi_{\varsigma})d\mu(\varsigma)\\
&= S_{\chi}^{\mathcal{J}}+S_{\xi}^{\mathcal{J}^{C}}+S_{\chi'}^{\mathcal{J}}+S_{\xi'}^{\mathcal{J}^{C}}+S_{\chi,\chi'}^{\mathcal{J}}+S_{\chi',\chi}^{\mathcal{J}}+ S_{\xi,\xi'}^{\mathcal{J}^{C}}+S_{\xi',\xi}^{\mathcal{J}^{C}}\\
&\geq ALL^{\ast}+A'LL^{\ast}\\
&=(A+A')LL^{\ast}.
\end{align*}
Also, $ \{\chi_{\varsigma}+\chi'_{\varsigma}\}_{\varsigma\in \mathcal{J}}\cup\{\xi_{\varsigma}+\xi'_{\varsigma}\}_{\varsigma\in \mathcal{J}^{C}} $ is a $ g- $Bessel sequence.
\end{proof}

\begin{definition} Let $ \chi=\{\chi_{\varsigma}\}_{\varsigma\in\mathfrak{A}} $ and $ \xi=\{\xi_{\varsigma}\}_{\varsigma\in\mathfrak{A}} $ be $ g- $Bessel sequences. 
	\begin{itemize}
		\item [(1)] $ \xi $ is  called an  $ L- $dual of $ \chi $ if for each $ f\in\mathcal{H} $, we have $ Lf= S_{\chi,\xi}(f)= \int_{\mathfrak{A}} \xi_{\varsigma}^{\ast}\chi_{\varsigma}(f)d\mu(\varsigma)$.
		\item[(2)] $ \xi $ is called an approximate $ L- $dual of $ \chi $ if there exists $ 0<t<1 $ such that for every $ f\in\mathcal{H} $
			\begin{equation*}
			\Vert L(f)-S_{\chi,\xi}(f)\Vert \leq t\Vert L(f)\Vert.
		\end{equation*}  
	\end{itemize}
	\end{definition}

\begin{proposition} Let $ \xi=\{\xi_{\varsigma}\}_{\varsigma\in\mathfrak{A}} $ be an approximate $ L- $dual of $ \chi $.
Then $ \chi $ has an $ L- $dual and every element $ L(f) $ of $ R(L) $ can be reconstructed from $\{ \xi_{\varsigma}^{\ast}\circ \chi_{\varsigma}(f) \}_{\varsigma\in\mathfrak{A}}$. 
	\end{proposition}

\begin{proof} Since $ \xi $ is an approximate $ L- $dual of $ \chi $, there exists $ 0<t<1 $ such that 
		\begin{equation}\label{1}
		\Vert L(f)-S_{\chi,\xi}(f)\Vert \leq t\Vert L(f)\Vert, \;\; f\in\mathcal{H}.
	\end{equation}
Then $ S_{\chi,\xi}(f)=0 \Leftrightarrow L(f)=0 $. Therefore, we get
$$ V:R(L)\Rightarrow R(S_{\xi,\chi}), $$ 
which is defined by
 $$  V(L(f))= S_{\xi,\chi}(f) ,\;\;\forall f\in\mathcal{H}.
 $$
  Hence $ V $ is an injective bounded linear map. By (\ref{1}),  we obtain
 \begin{equation}
 		\Vert L(f)-V(L(f))\Vert \leq t\Vert L(f)\Vert, \;\; f\in\mathcal{H}
 \end{equation}
and  so, for every $ f\in \mathcal{H} $,
 \begin{equation*}
 (1-t)	\Vert Lf\Vert \leq\Vert V(L(f))\Vert \leq (1+t)\Vert Lf\Vert.
 \end{equation*}
Hence $ V $ has a closed range, $ R(V)= R(S_{\xi,\chi}) $ and we can extend
 $ V^{-1} : R(S_{\xi,\chi})\rightarrow R(L)$ to $ U: \mathcal{H}\rightarrow \mathcal{H} $ by $ U= V^{-1}\circ \pi_{R(L)}  $, where $ \pi_{R(L)} $ is the orthogonal projection of $ \mathcal{H} $ onto $ R(V) $. Then
 \begin{equation*}
 L(f)= U\circ S_{\xi,\chi}(f)=\int_{\mathfrak{A}}(U \circ \xi_{\varsigma}^{\ast})\circ \chi_{\varsigma}(f)d\mu(\varsigma), \;\;f\in\mathcal{H}.
 \end{equation*} 
Thus  $ \{\xi_{\varsigma}\circ U^{\ast}\}_{\varsigma\in\mathfrak{A}} $ is an $ L- $dual of $\{ \chi_{\varsigma}\}_{\varsigma\in\mathfrak{A}} $. 
\end{proof}

Now, we study the continuous $ L-g-  $frames under some perturbations and we try to show that approximate $ L- $duals are stable under small perturbation in a continuous case.

\begin{theorem} Let $ \chi=\{\chi_{\varsigma}\in B(\mathcal{H},\mathcal{H_{\varsigma}})\}_{\varsigma\in\mathfrak{A}} $ be a continuous $ g- $Bessel sequence and $ \Phi= \{\phi \in B(\mathcal{H},\mathcal{H_{\varsigma}})\}_{\varsigma\in\mathfrak{A}} $ be an approximate $ L- $dual (resp. $ L- $dual) of $ \chi $ with $ 0<t<1 $ and upper bound $ D $. If $ \xi=\{\xi_{\varsigma}\}\in B(\mathcal{H},\mathcal{H_{\varsigma}})\}_{\varsigma\in\mathfrak{A}} $  is a sequence such that 
	\begin{equation*}
		(\int_{\mathfrak{A}} \Vert (\chi_{\varsigma}-\xi_{\varsigma})(f)\Vert^{2}d\mu(\varsigma))^{\frac{1}{2}}\leq  G \Vert Lf\Vert,\;\;f\in\mathcal{H},
	\end{equation*}
	and $ \sqrt{DG}< 1-t $ (resp. $ DG<1 $), then $ \Phi $ is an approximate $ L- $dual of $ \xi $.
	\end{theorem}

\begin{proof}Let $ B $ be an upper bound for $ \chi $. Then we have 
	\begin{equation*}
	(	\int_{\mathfrak{A}}\Vert \xi_{\varsigma} f\Vert^{2}d\mu(\varsigma))^{\frac{1}{2}}\leq (\int_{\mathfrak{A}}\Vert\chi_{\varsigma} f\Vert^{2} d\mu(\varsigma))^{\frac{1}{2}}+ (\int_{\mathfrak{A}}\Vert \xi_{\varsigma}f-\chi_{\varsigma}f\Vert^{2} d\mu(\varsigma))^{\frac{1}{2}}\leq (\sqrt{B}+\sqrt{G}\Vert L\Vert)\Vert f\Vert.
	\end{equation*} 
	This implies that $ \xi $ is a $ g- $Bessel sequence. 
	
	We have 
	\begin{equation*}
		\Vert S_{\Phi,\chi}f-S_{\Phi,\xi}f\Vert \leq sup_{\Vert g\Vert=1}\{ 	(\int_{\mathfrak{A}} \Vert (\chi_{\varsigma}-\xi_{\varsigma})(f)\Vert^{2}d\mu(\varsigma))^{\frac{1}{2}} (\int_{\mathfrak{A}}\Vert \phi g\Vert^{2}d\mu(\varsigma))^{\frac{1}{2}}\}\leq \sqrt{DG}\Vert Lf\Vert.
	\end{equation*}
	Then 
	\begin{equation*}
		\Vert Lf-S_{\Phi, \xi}f\Vert \leq \Vert Lf-S_{\Phi, \chi}f\Vert+ \Vert S_{\Phi,\chi}f-S_{\Phi, \xi}f\Vert\leq (t+\sqrt{DG})\Vert Lf \Vert.
	\end{equation*}
Since $ t+\sqrt{DG}<1 $, we obtain the result.
\end{proof}

\begin{theorem} Let $ \chi=\{\chi_{\varsigma}\in B(\mathcal{H},\mathcal{H_{\varsigma}})\}_{\varsigma\in\mathfrak{A}} $ and $ \xi=\{\xi_{\varsigma}\in B(\mathcal{H},\mathcal{H_{\varsigma}})\}_{\varsigma\in\mathfrak{A}} $  be $ (A,\;B) $ woven continuous $ L-g- $frames and let $ K\in B(\mathcal{H}) $ and $ K_{\varsigma} ,\;K'_{\varsigma} \in B(\mathcal{H}_{\varsigma})$.  If there exist $ 0< \alpha< \beta< \infty $ such that for each $ \varsigma \in\mathfrak{A} $ and $ F_{\varsigma} \in\mathcal{H_{\varsigma}}$, $ \alpha \Vert F_{\varsigma}\Vert \leq \Vert K_{\varsigma} f\Vert $, $ \Vert K'_{\varsigma}F_{\varsigma} \Vert \leq \beta \Vert F_{\varsigma}\Vert$, then $ \{\chi'_{\varsigma}=K_{\varsigma}\chi_{\varsigma}K\}_{\varsigma\in\mathfrak{A}} $ and  $ \{\xi_{\varsigma}= K'_{\varsigma}\xi'_{\varsigma}K\}_{\varsigma\in\mathfrak{A}} $ are woven continuous $ T^{\ast}L-g $ frames with universal bounds $ \alpha^{2}A $ and $ \beta^{2} B \Vert K \Vert^{2}$.
	
Moreover, if $ KL^{\ast}= L^{\ast}K $, $ \alpha \Vert F\Vert \leq \Vert KF\Vert  $, then  $ \chi'=\{\chi'_{\varsigma}\in B(\mathcal{H},\mathcal{H_{\varsigma}})\}_{\varsigma\in\mathfrak{A}} $ and $ \xi'=\{\xi'_{\varsigma}\in B(\mathcal{H},\mathcal{H_{\varsigma}})\}_{\varsigma\in\mathfrak{A}} $  are woven continuous $ L-g- $ frames with universal bounds $ \alpha^{4}A $ and $ \beta^{2}B \Vert K\Vert ^{2}$.
	\end{theorem}

\begin{proof}For every $ J\subset \mathfrak{A} $, we have
	\begin{align*}
		\int_{J}\Vert\chi'_{\varsigma}F\Vert^{2}d\mu(\varsigma)+\int_{J^{C}}\Vert\xi'_{\varsigma}F\Vert^{2}d\mu(\varsigma)&=	\int_{J}\Vert K_{\varsigma}\chi_{\varsigma}KF\Vert^{2}d\mu(\varsigma)+\int_{J^{C}}\Vert K' \xi'_{\varsigma}KF\Vert^{2}d\mu(\varsigma)\\
		&\leq 	\int_{J}\Vert K_{\varsigma}\Vert^{2} \Vert\chi_{\varsigma}KF\Vert^{2}d\mu(\varsigma)+\int_{J^{C}}\Vert K'\Vert^{2}\Vert  \xi'_{\varsigma}KF\Vert^{2}d\mu(\varsigma)\\
		&\leq \beta^{2}(\int_{J} \Vert\chi_{\varsigma}KF\Vert^{2}d\mu(\varsigma)+\int_{J^{C}}\Vert  \xi'_{\varsigma}KF\Vert^{2}d\mu(\varsigma))\\
		&\leq \beta^{2}B\Vert K\Vert^{2} \Vert F \Vert^{2},\;\;\forall F\in\mathcal{H}.
	\end{align*}
Similarly, we obtain 
\begin{equation*}
	\int_{J}\Vert\chi'_{\varsigma}F\Vert^{2}d\mu(\varsigma)+\int_{J^{C}}\Vert\xi'_{\varsigma}F\Vert^{2}d\mu(\varsigma)\geq \alpha^{2}A\Vert (K^{\ast}L)^{\ast}F \Vert ^{2}, \;\;\;\forall F\in\mathcal{H}.
\end{equation*}
This completes the proof.
	\end{proof}

\begin{corollary}Let $ \chi=\{\chi_{\varsigma}\in B(\mathcal{H},\mathcal{H_{\varsigma}})\}_{\varsigma\in\mathfrak{A}} $ be an $ L-g- $frame for $ \mathcal{H} $ and $ K \in B(\mathcal{H})$ be invertible. Then
\begin{itemize}
	\item[(1)] $ \{\chi_{\varsigma}K\}_{\varsigma} $ is a continuous $ L-g- $frame, when $ \xi L^{\ast} = L^{\ast}\xi.$
	\item[(2)]$ \{K\chi_{\varsigma}\}_{\varsigma} $ is a continuous $ L-g- $frame, when $ \mathcal{H}_{\varsigma} \subset \mathcal{H}$, $\forall \varsigma \in\mathfrak{A}.  $
\end{itemize}
	\end{corollary}
\begin{proof}Let $ \{\chi_{\varsigma}\}_{\varsigma \in \mathfrak{A}} $ be a continuous $ L-g- $frame with bounds $ A $ and $ B $. 
	 
$ (1) $ For every $ f\in\mathcal{H} $, we have
\begin{align*}
\dfrac{A}{\Vert K^{-1}\Vert ^{2}} \Vert L^{\ast}f \Vert^{2}&\leq A\Vert KL^{\ast}f \Vert^{2}= A\Vert L^{\ast}Kf\Vert^{2}\\
&\leq \int_{\mathfrak{A}}\Vert \chi_{\varsigma} Kf \Vert^{2} d\mu(\varsigma)\\
&\leq B\Vert Kf\Vert^{2}\\
& \leq B\Vert K\Vert ^{2} \Vert f\Vert^{2}
\end{align*}
 $ (2) $ We have
\begin{align*}
	\dfrac{A}{\Vert K^{-1}\Vert ^{2}} \Vert L^{\ast}f \Vert^{2}&\leq \dfrac{1}{\Vert K^{-1}\Vert^{2}}\int_{\mathfrak{A}}\Vert K\chi_{\varsigma} f \Vert^{2} d\mu(\varsigma)\\
	&\leq\Vert K\Vert^{2} \int_{\mathfrak{A}}\Vert \chi_{\varsigma} f \Vert^{2} d\mu(\varsigma)\\
	& \leq B\Vert K\Vert ^{2} \Vert f\Vert^{2},\;\;\forall f\in \mathcal{H}.
\end{align*}
This completes the proof.
\end{proof}

	\begin{proposition} Let $ \{\chi_{\varsigma}\}_{\varsigma \in\mathfrak{A}} $ and $ \{\xi_{\varsigma}\}_{\varsigma \in\mathfrak{A}} $ be $ (A,B) $ woven continuous $ L-g- $frames. 
		
	If 
		$$ \int_{J}\Vert \chi_{\varsigma}f\Vert^{2}d\mu(\varsigma)\leq C \Vert L^{\ast} f\Vert ^{2}, \;\;\forall f\in\mathcal{H}$$ 
	for some $ 0<C<A $, then  $ \{\chi_{\varsigma}\}_{\varsigma \in \mathcal{J}^{C}} $ and $ \{\xi_{\varsigma}\}_{\varsigma \in \mathcal{J}^{C}}$  are $ (A-C,B) $ woven $ L-g- $frames.
	\end{proposition}

\begin{proof}Suppose that $ E\subset \mathfrak{A} \backslash \mathcal{J} $.
	Then  for all $f\in\mathcal{H} $, 
	\begin{align*}
	&	\int_{E}\vert \langle f,\chi_{\varsigma}\rangle \vert^{2}d\mu(\varsigma)+ \int_{E^{C}}\vert \langle f,\xi_{\varsigma}\rangle \vert^{2}d\mu(\varsigma)\\
	&=(\int_{E\cup \mathcal{J}}\vert \langle f,\chi_{\varsigma}\rangle \vert^{2}d\mu(\varsigma)-\int_{ \mathcal{J}}\vert \langle f,\chi_{\varsigma}\rangle \vert^{2}d\mu(\varsigma))+\int_{E^{C}}\vert \langle f,\xi_{\varsigma}\rangle \vert^{2}d\mu(\varsigma)\\
	&\geq (A-C)\Vert f\Vert^{2},
	\end{align*}  
and so a lower weaving bound is $ A-C $. 	
\end{proof}

\begin{corollary} Let  $ \{\chi_{\varsigma}\}_{\varsigma \in\mathfrak{A}} $ be a continuous $ L-g- $frame with lower frame bound $ A $. If for some $ \mathcal{J}\subset \mathfrak{A} $ and $ 0<C<A $, 
	$$ \int_{\mathcal{J}}\Vert \chi_{\varsigma}f\Vert^{2}d\mu(\varsigma)\leq C \Vert L^{\ast} f\Vert ^{2}, \;\;\forall f\in\mathcal{H},$$ 
	then $ \{\chi_{\varsigma}\}_{\varsigma \in \mathcal{J}^{C}} $ is a continuous $ L-g- $frame with lower bound $ A-C $. 
\end{corollary}

\begin{definition}Let $ \{\chi_{\varsigma}\}_{\varsigma\in \mathcal{J}} $ be a continuous $ L-g- $frame and let $ 0\leq \alpha_{1},\alpha_{2}<1 $. We say that the family $ \{\xi_{\varsigma}\}_{\varsigma \in \mathfrak{A}} $ is an $( \alpha_{1},\alpha_{2}) -$ perturbation of $ \{\chi_{\varsigma}\}_{\varsigma\in \mathfrak{A}} $ if we have 
	\begin{equation*}
		\Vert \chi_{\varsigma}f -\xi_{\varsigma} f\Vert \leq \alpha_{1} \Vert \chi_{\varsigma}f\Vert +\alpha_{2} \Vert \xi_{\varsigma}f\Vert,\;\;\;\forall f\in\mathcal{H}.
	\end{equation*}
	\end{definition}

\begin{theorem}\label{t3} Let $ \{\chi_{\varsigma}\}_{\varsigma \in\mathfrak{A}} $ and $ \{\xi_{\varsigma}\}_{\varsigma \in\mathfrak{A}} $ be woven continuous  $ L-g-$frames and $ \{\chi'_{\varsigma}\}_{\varsigma \in\mathfrak{A}} $, $ \{\xi'_{\varsigma}\}_{\varsigma \in\mathfrak{A}} $ be $ (\alpha_{1},\alpha_{2}) ,\;(\beta_{1},\beta_{{2}})$-perturbations of $ \{\chi_{\varsigma}\}_{\varsigma \in\mathfrak{A}} $ and $ \{\xi_{\varsigma}\}_{\varsigma \in\mathfrak{A}} $, respectively. Then $ \{\chi'_{\varsigma}\}_{\varsigma \in\mathfrak{A}} $ and 
$ \{\xi'_{\varsigma}\}_{\varsigma \in\mathfrak{A}} $ are woven $ L-g- $frames.	
	\end{theorem}

\begin{proof} For $ f\in\mathcal{H} $, we have
	\begin{equation*}
		\Vert \chi'_{\varsigma}f \Vert-\Vert \chi_{\varsigma}f \Vert\leq \Vert  \chi'_{\varsigma}f - \chi_{\varsigma}f \Vert \leq \alpha_{1} \Vert \chi_{\varsigma}f \Vert +\alpha_{2}\Vert \chi'_{\varsigma}f\Vert .
	\end{equation*}
So 
\begin{equation*}
	\dfrac{1-\alpha_{1}}{1+\alpha_{2}}\Vert \chi_{\varsigma}f\Vert \leq \Vert \chi'_{\varsigma}f\Vert \leq \dfrac{1+\alpha_{1}}{1-\alpha_{2}}\Vert \chi_{\varsigma}f\Vert . 
\end{equation*}	

Similarly, we have 
\begin{equation*}
	\dfrac{1-\beta_{1}}{1+\beta_{2}}\Vert \xi_{\varsigma}f\Vert \leq \Vert \xi'_{\varsigma}f\Vert \leq \dfrac{1+\beta_{1}}{1-\beta_{2}}\Vert \xi_{\varsigma}f\Vert.  
	\end{equation*}
For $ \mathcal{J}\subset \mathfrak{A} $ and for all $f\in \mathcal{H} $,
\begin{align*}
	&min\{(\dfrac{1-\alpha_{1}}{1+\alpha_{2}})^{2},(\dfrac{1-\beta_{1}}{1+\beta_{2}})^{2}\}(\int_{\mathcal{J}}\Vert \chi_{\varsigma}f\Vert ^{2}d\mu(\varsigma)+\int_{\mathcal{J}^{C}}\Vert \xi_{\varsigma}f\Vert^{2}d\mu(\varsigma))\\
&\leq \int_{\mathcal{J}}\Vert \chi_{\varsigma}f\Vert ^{2}d\mu(\varsigma)+\int_{\mathcal{J}^{C}}\Vert \xi_{\varsigma}f\Vert^{2}d\mu(\varsigma)\\
&\leq max\{(\dfrac{1-\alpha_{1}}{1+\alpha_{2}})^{2},(\dfrac{1-\beta_{1}}{1+\beta_{2}})^{2}\}(\int_{\mathcal{J}}\Vert \chi_{\varsigma}f\Vert ^{2}d\mu(\varsigma)+\int_{\mathcal{J}^{C}}\Vert \xi_{\varsigma}f\Vert^{2}d\mu(\varsigma)).
\end{align*}
This completes the proof.
\end{proof}

\begin{corollary} Let $ \{\chi_{\varsigma}\}_{\varsigma \in\mathfrak{A}} $ and $ \{\xi_{\varsigma}\}_{\varsigma \in\mathfrak{A}} $ be woven continuous  $ L-g -$frames and $ \{\chi'_{\varsigma}\}_{\varsigma \in\mathfrak{A}} $, $ \{\xi'_{\varsigma}\}_{\varsigma \in\mathfrak{A}} $ be sequences and $ 0\leq m_{1},m_{2} $ such that 
 for all $f\in\mathcal{H}$ and all $ \varsigma \in\mathfrak{A},
	 $
	 \begin{align*}
	 &	\Vert \chi_{\varsigma} f-\chi'_{\varsigma}f\Vert \leq m_{1}\;min\{\Vert\chi_{\varsigma}f\Vert, \Vert \chi'_{\varsigma}f\Vert\},\\
	 &\Vert \xi_{\varsigma} f-\xi'_{\varsigma}f\Vert \leq m_{2}\;min\{\Vert\xi_{\varsigma}f\Vert, \Vert \xi'_{\varsigma}f\Vert\}.
	 \end{align*}
Then $ \{\chi'_{\varsigma}\}_{\varsigma \in\mathfrak{A}} $ and $ \{\xi'_{\varsigma}\}_{\varsigma \in\mathfrak{A}} $ are woven continuous $ L-g- $frames.	
\end{corollary}

\begin{proof}For all $\varsigma \in\mathfrak{A} $ and $ f\in\mathcal{H} $, we have 
	\begin{equation*}
		\dfrac{1}{1+m_{1}}\Vert \chi_{\varsigma}f\Vert \leq \Vert \chi'_{\varsigma}f\Vert \leq (m_{1}+1)\Vert \chi_{\varsigma}f\Vert, 	
	\end{equation*}	
	\begin{equation*}
		\dfrac{1}{1+m_{2}}\Vert \xi_{\varsigma}f\Vert \leq \Vert \xi'_{\varsigma}f\Vert \leq (m_{2}+1)\Vert \xi_{\varsigma}f\Vert. 
	 	\end{equation*}
	By a similar method  to the proof of Theorem \ref{t3},  we have the result.
\end{proof}

\section{Conclusion}

 In this article, we have introduced the atomic $g$-system and we have generalized some of the known results in continuous $L$-frames, weaving continuous and weaving continuous $ g$-frames, also we have studied weaving continuous $ L$-$g$-frames in Hilbert spaces. 
Furthermore, we have  studied the behaviour continuous $ L$-$g$-frames under some perturbations and  we have showed  that approximate $L$-duals are stable under small perturbation and that it is possible to remove some elements of a woven continuous $ L$-$g$-frame and still have a woven continuous $ L$-$g$-frame.

\medskip

\section*{Declarations}

\medskip

\noindent \textbf{Availablity of data and materials}\newline
\noindent Not applicable.

\medskip

\noindent \textbf{Competing  interest}\newline
\noindent The authors declare that they have no competing interests.

\medskip

\noindent \textbf{Fundings}\newline
\noindent  Authors declare that there is no funding available for this article.

\medskip

\noindent \textbf{Authors' contributions}\newline
\noindent The authors equally conceived of the study, participated in its
design and coordination, drafted the manuscript, participated in the
sequence alignment, and read and approved the final manuscript. 

\medskip

\noindent \textbf{Acknowledgements}\newline
\noindent We would like to express our sincere gratitude to the anonymous referee for his/her helpful comments that will help to improve the quality of the manuscript.

\medskip

\end{document}